\documentclass[12pt,amstex]{amsart}

\usepackage{amsmath, amsthm, amssymb, mathrsfs, enumerate, bm, enumitem}
\usepackage{comment}
\usepackage{color}
\usepackage[dvipdfmx]{graphicx}
\usepackage{xparse}
\usepackage{mathtools}

\usepackage{a4wide}
\newtheorem{theorem}{Theorem}[section]
\newtheorem{corollary}[theorem]{Corollary}

\newtheorem{lemma}[theorem]{Lemma}
\newtheorem{proposition}[theorem]{Proposition}

\theoremstyle{definition}

\theoremstyle{remark}

\DeclareMathOperator{\Z}{\mathbb{Z}}

\DeclareMathOperator{\R}{\mathbb{R}}
\DeclareMathOperator{\C}{\mathbb{C}}

\DeclareMathOperator{\Aut}{Aut}

\DeclareMathOperator{\e}{\mathrm{e}}

\DeclareMathOperator{\I}{\mathrm{I}}

\DeclareMathOperator{\m}{\mathrm{m}}
\DeclareMathOperator{\degg}{\mathrm{deg}}

\DeclareMathOperator{\cbc}{\mathrm{cbc}}
\DeclareMathOperator{\bc}{\mathrm{bc}}

\DeclarePairedDelimiterX\set[1]\lbrace\rbrace{\def\given{\;\delimsize\vert\;}#1}

\begin{document}
\title[Bartholdi zeta function formula for graphs with bounded degree]
{A generalized Bartholdi zeta function formula for simple graphs with bounded degree}

\author[T.~Kousaka]
{Taichi Kousaka}

\address{Graduate School of Mathematics, Kyushu University,
744 Motooka, Nishi-ku, Fukuoka, 819-0395 JAPAN}
\email{t-kosaka@math.kyushu-u.ac.jp}

\keywords{Bartholdi zeta function, Ihara zeta function, determinant formula, heat kernel, Bessel function.}
\subjclass[2010]{05C38, 05C63, 11M36}

\begin{abstract}
We introduce a generalized Bartholdi zeta function for simple graphs with bounded degree. 
This zeta function is a generalization of both the Bartholdi zeta function which was introduced by L.~Bartholdi and the Ihara zeta function which was introduced by G.~Chinta, J.~Jorgenson and A.~Karlsson. 
Furthermore, we establish a Bartholdi type formula of this Bartholdi zeta function for simple graphs with bounded degree. 
Moreover, for regular graphs, we give a new expression of the heat kernel which is regarded as a one-parameter deformation of the expression obtained by G.~Chinta, J.~Jorgenson and A.~Karlsson. 
By applying this formula, we give an alternative proof of the Bartholdi zeta function formula for regular graphs. 
\end{abstract}
\maketitle
\section{Introduction}
All graphs in this paper are assumed to be connected, countable and simple. 
Let $X$ be a graph with bounded degree and $\Delta_{X}$ be the combinatorial Laplacian of $X$. 
It is well-known that the spectrum of $\Delta_{X}$ is closely related to geometric properties and combinatorial properties of $X$ at least from the view point of graph theory, number theory and probability theory. 
Classically, it is important to study the relationship between closed paths in $X$ and the spectrum of $\Delta_{X}$. 
In this paper, we study the relationship from the view point of number theory. 

\medskip
Especially for a finite graph $X$, it is well-known that closed geodesics of $X$ are deeply related to the spectrum of $\Delta_{X}$ (cf.~\cite{Terras}). 
The relationship is described as the Ihara formula explicitly. 
The Ihara zeta function for a finite graph $X$ is defined by
\begin{align*}
Z_{X}(u)=\exp\bigg( \sum_{m=1}^{\infty}\frac{N_{m}}{m}u^{m} \bigg). 
\end{align*}
Here, $N_{m}$ stands for the number of closed geodesics of length $m$ in $X$. 
Then, the Ihara formula is described as follows (cf.~\cite{Terras}). 
\begin{align*}
Z_{X}(u)^{-1}=(1-u^{2})^{-\chi(X)} \det \big( I-u(D_{X}-\Delta_{X})+u^{2}(D_{X}-I) \big). 
\end{align*}
Here, $\chi(X)$ stands for the Euler characteristic of $X$ and $D_{X}$ stands for the valency operator of $X$. 
For finite regular graphs, the above formula was originally established by Y.~Ihara in the $p$-adic setting (\cite{YI1966}). 
Then, it has been generalized by T.~Sunada, K.~Hashimoto and H.~Bass (\cite{Bass1992}, \cite{HH89}, \cite{H89}, \cite{H90}, \cite{H92}, \cite{H93}, \cite{KS2000}, \cite{serre1}, \cite{sunada86}, \cite{sunada88}). 
When $X$ is regular, this formula gives an explicit relationship between the number of closed geodesic and the spectrum of $\Delta_{X}$. 

In $1999$, L.~Bartholdi introduced the Bartholdi zeta function for finite graphs and established a determinant expression of it (\cite{Bartholdi1999}). 
The Bartholdi zeta function is defined by 
\begin{align*}
Z_{X}(u,t)=\exp\bigg(\sum_{C \in \mathcal{C}} \frac{1}{\ell(C)}t^{\cbc(C)}u^{\ell(C)} \bigg). 
\end{align*}
Here, we denote by $\mathcal{C}$ the set of closed paths in $X$, by $\ell(C)$ the length of $C$ and by $\cbc(C)$ the cyclic bump count of a closed path $C$. 
This is a generalization of the Ihara zeta function by adding a variable $t$ which plays a role of counting back-trackings of a closed path. 
Indeed, if $t$ is equal to $0$, this zeta function coincides with the Ihara zeta function $Z_{X}(u)$. 
The determinant expression of $Z_{X}(u, t)$ is described as follows (\cite{Bartholdi1999}). 
\begin{align*}
Z_{X}(u, t)^{-1}&=\big(1-(1-t)^{2}u^{2}\big)^{-\chi(X)}\\
&\quad \times \det \big(I-u(D_{X}-\Delta_{X})+(1-t)u^{2}(D_{X}-(1-t)I) \big). 
\end{align*}
As in the case of the Ihara formula, when $X$ is regular, this formula gives the explicit relationship between the number of closed paths and the spectrum of $\Delta_{X}$. 

\medskip
Recently, several generalizations of the Ihara zeta function from finite graphs to infinite graphs have been considered (cf.~\cite{CJK1}, \cite{CMS2001}, \cite{Cla2009}, \cite{Deit2015}, \cite{GZ2004}, \cite{GIL2008a}, \cite{GIL2008b}, \cite{Sch1999}). 
In this paper, we follow \cite{CJK1} essentially. 
In 2017, the author introduced the Ihara zeta function for a graph $X$ with bounded degree as follows (\cite{K2017}). 
\begin{align*}
Z_{X}(u, x_{0})=\bigg( \sum_{m=1}^{\infty}\frac{N_{m}(x_{0})}{m}u^{m} \bigg). 
\end{align*}
Here, $N_{m}(x_{0})$ stands for the number of closed geodesics of length $m$ starting at $x_{0}$. 
If $X$ is vertex-transitive, this zeta function coincides with the Ihara zeta function which was introduced in \cite{CJK1}. 
In \cite{CJK1}, the definition of the Ihara zeta function for regular graphs is given (p.~185 in \cite{CJK1}). 
In general, however, when $X$ is regular, the Ihara zeta function in \cite{CJK1} does not always coincide with our Ihara zeta function. 
In \cite{CJK1}, G.~Chinta, J.~Jorgenson and A.~Karlsson established the Ihara type formula for the Ihara zeta function for vertex-transitive graphs by giving a new expression of the heat kernel (\cite{CJK1}). 
This definition works well in the point of studying deeply the relationship between closed geodesics and the spectrum of $\Delta_{X}$ and also as an analogy with heat kernel analysis of rank one symmetric spaces. 
After that, the author established the Ihara type formula for the Ihara zeta function for graphs with bounded degree (\cite{K2017}). 
His proof also gives an alternative proof of the formula for vertex-transitive graphs. 

\medskip
In this paper, we study the relationship between closed paths and the spectrum of $\Delta_{X}$ by introducing a Bartholdi zeta function for graphs with bounded degree. 
For a graph $X$ with bounded degree and a vertex $x_{0}$, a Bartholdi zeta function is defined by as follows in this paper. 
\begin{align*}
Z_{X}(u, t, x_{0})=\exp \bigg( \sum_{C \in \mathcal{C}_{x_{0}}}\frac{1}{\ell(C)}t^{\cbc(C)}u^{\ell(C)} \bigg). 
\end{align*}
Here, we denote by $\mathcal{C}_{x_{0}}$ the set of closed paths starting at $x_{0}$. 
We remark that we introduce a Bartholdi zeta function which is a generalization of the above. 
However, we do not introduce the definition in Introduction because the definition is a little technical in the sense of being based on a path counting formula. 
If $t$ is equal to $0$, this Bartholdi zeta function coincides with $Z_{X}(u, x_{0})$. 
If $X$ is a finite graph, by the definition of $Z_{X}(u, t, x_{0})$, the following equality holds. 
\begin{align*}
\prod_{x_{0} \in V\!X} Z_{X}(u, t, x_{0}) = Z_{X}(u, t). 
\end{align*}
In this sense, this Bartholdi zeta function is a generalization of the original one. 
Furthermore, we present a Bartholdi type formula for this Bartholdi zeta function. 
Especially for regular graphs, this formula describes the relationship between the number of closed paths and the spectrum of $\Delta_{X}$. 
We remark that for finite graphs, this formula can be regarded as a refined version of the original Bartholdi zeta function formula. 

Moreover, for (possibly infinite) regular graphs, we give a new expression of the heat kernel which is regarded as a one-parameter deformation of the expression obtained in \cite{CJK1}. 
By applying this formula, we give an alternative proof of the Bartholdi formula for regular graphs. 
We note that our heat kernel approach to the Bartholdi formula is new even for finite regular graphs. 
This is an important application of our new heat kernel expression. 
Many applications of the heat kernel are well-known. 
Therefore, in addition to the above application, we believe that there should be more applications by using our new expression of the heat kernel. 

\medskip
This paper is organized as follows. 
In Section $2$, we prepare some terminologies. 
In Section $3$, we give a path counting formula which is a generalization of the formula obtained in \cite{K2017}. 
In Section $4$, we give a Bartholdi type formula for our Bartholdi zeta function. 
In Section $5$, we give the Euler product expression of it. 
In Section $6$, we give a new expression of the heat kernel for a regular graph (not necessarily finite). 
In Section $7$, as an application of our heat kernel expression, we give an alternative proof of the Bartholdi type formula obtained in Section $4$. 
 
\section{Preliminaries}
\subsection{Graphs and Paths}
In this section, we give terminology of graphs and paths used throughout this paper (cf.~\cite{Bartholdi1999}, \cite{serre1}, \cite{sunada1}). 
A graph $X$ is an ordered pair $(V\!X, E\!X)$ of disjoint sets $V\!X$ and $E\!X$ with two maps, 
\begin{align*}
E\!X \rightarrow V\!X \times V\!X, e \mapsto (o(e), t(e)), \quad
E\!X \rightarrow E\!X , e \mapsto \bar{e}
\end{align*}
such that for each $e \in E\!X$, $\bar{e} \neq e$, $\bar{\bar{e}} = e$, $o(e)=t(\bar{e})$. 
For a graph $X=(V\!X, E\!X)$, two sets $V\!X$ and $E\!X$ are called vertex set and edge set respectively. 
A graph $X$ is {\it simple} if $X$ has no loops and multiple edges. 
For a vertex $x \in V\!X$, the {\it degree of $x$} is the cardinality of the set $E_{x}$, where $E_{x}=\{ e \in E\!X | o(e)=x\}$. We denote the degree of $x$ by $\degg(x)$. 
A graph $X$ is {\it countable} if the vertex set is countable. 
A graph $X$ {\it has bounded degree} if the supremum of the set of all degrees is not infinite. 
For a graph $X$, a {\it path of length $n$} is a sequence of edges
\begin{align*}
C=(e_{1}, \dots, e_{n})
\end{align*}
such that $t(e_{i})=o(e_{i+1})$ for each $i$. 
We denote $o(e_{1})$ by $o(C)$, $t(e_{n})$ by $t(C)$ and the length of $C$ by $\ell(C)$. 
A path $C$ is {\it closed} if $o(C)=t(C)$. 
We regard a vertex as a path of length $0$. 
A path $C=(e_{1}, \dots, e_{n})$ {\it has a back-tracking or bump} if there exist $i$ such that $e_{i+1}=\bar{e_{i}}$. A path $C=(e_{1}, \dots, e_{n})$ {\it has a tail} if $e_{n}=\bar{e_{1}}$. 
For a path $C=(e_{1}, \dots, e_{n})$, we define the {\it bump count} of $C$ as follows. 
\begin{align*}
\bc(C)=\sharp \set{ i \in \{1, \dots, n-1\} \given e_{i}=\overline{e_{i+1}} }. 
\end{align*}
For a closed path $C=(e_{1}, \dots, e_{n})$, we define the {\it cyclic bump count} of $C$ as follows. 
\begin{align*}
\cbc(C)=\sharp \set{ i \in \Z/m\Z \given e_{i}=\overline{e_{i+1}} }. 
\end{align*}
For a closed path $x_{0}$, we define $\bc(x_{0})=\cbc(x_{0})=0$. 
For a path $C=(e_{1}, \dots, e_{m})$, we denote $e_{i}$ by $e_{i}(C)$. 

\subsection{The Laplacian of a graph}
For the vertex set $V\!X$ of a graph $X$, we define {\it the $\ell^2$-space on the vertex set $V\!X$} by 
\begin{align*}
\ell^{2}(V\!X)=\set[\bigg]{ f \colon V\!X \rightarrow \C \given \sum_{x \in V\!X}|f(x)|^2 < +\infty }. 
\end{align*}
For a function $f \in \ell^2(V\!X)$ and a vertex $x \in V\!X$, we define the {\it adjacency operator} $A_{X}$ on $X$ and the {\it valency operator} $D_{X}$ on $X$ as follows respectively. 
\begin{align*}
(A_{X}f)(x)=\sum_{e \in E_{x}}f(t(e)), \\
(D_{X}f)(x)=\degg(x)f(x). 
\end{align*}
Then,  we define the {\it Laplacian} $D_{X}$ on $X$ by $\Delta_{X}=D_{X}-A_{X}$. 
The Laplacian is a semipositive and self-adjoint bounded operator if $X$ has bounded degree. 

\subsection{The heat kernel of a graph}
For a graph $X$ with bounded degree and a fixed vertex $x_{0}$, 
the {\it heat kernel} $K_{X}(\tau, x_{0}, x) \colon \R_{\geq 0} \times V\!X \to \R$ on $X$ is the solution of the heat equation 
\begin{align*}
	\left \{ \begin{array}{ll}
	\big(\Delta_{X}+\frac{\partial}{\partial \tau} \big)f(\tau, x)=0 ,\\
	f(0, x)=\delta_{x_{0}}(x). 
	\end{array} \right.
\end{align*}
Here, the function $f(\tau, x)$ is in the class $C^{1}$ on $\R \times V\!X$ for each $x \in V\!X$ and the function $\delta_{x_{0}}(x)$ is the Kronecker delta. 
The heat kernel on $X$ uniquely exists among functions which are bounded on $[0, T] \times V\!X$ for each $T \in \R_{\geq0}$ under our assumptions (\cite{dodziuk}). 
By the uniqueness of the solution of the heat equation, it turns out that the heat kernel $K_{X}(\tau, x_{0}, x)$ is an invariant under the automorphism group $\Aut(X)$. 

\subsection{The modified Bessel function}
In this section, we define the modified Bessel function and introduce some well-known properties of them. 
For $n \in \Z_{\geq 0}$ and $\tau \in \R$, we define the {\it modified Bessel function of the first kind} by the following power series. 
\begin{align*}
\I_{n}(\tau)=\sum_{m=0}^{\infty}\frac{\big( \tau/2 \big)^{n+2m}}{m! (m+n)!}. 
\end{align*}
For $-n \in \Z_{<0}$, we define $\I_{-n}(\tau)$ as follows. 
\begin{align*}
\I_{-n}(\tau)=\I_{n}(\tau). 
\end{align*}
It is well-known that $\I_{n}(\tau)$ is the power series solution of the following differential equation. 
\begin{align*}
\tau^{2} \frac{d^{2}w}{d\tau^{2}}+\tau \frac{dw}{d\tau}-(\tau^{2}+n^{2})w=0. 
\end{align*}
Moreover, it is also well-known that $\I_{n}(\tau)$ satisfies the following formula. 
\begin{align}\label{Bf}
2\frac{d}{d\tau}\I_{n}(\tau)=\I_{n-1}(\tau)+\I_{n+1}(\tau). 
\end{align}
In addition, for $n \geq 0$ and $\tau \in \R_{\geq 0}$, $\I_{n}(\tau)$ has the following trivial bound. 
\begin{align}\label{Bbound}
\I_{n}(\tau) \leq \bigg(\frac{\tau}{2}\bigg)^{n}\frac{\e^{\tau}}{n!}. 
\end{align}

\subsection{$G(t)$-transform}
For a real valued function $f(\tau)$$(0< \tau < \infty)$ which is integrable in every finite interval, we define $G(t)f$ as follows. 
\begin{align*}
G(t)f(u)=(u^{-2}-(q+t)(1-t))\int_{0}^{\infty}\e^{-\big( (q+t)(1-t)u+\frac{1}{u}-(q+1)\big)\tau}f(\tau)d\tau. 
\end{align*}
We call this transform {\it $G(t)$-transform}. 
The following formula holds (cf.~\cite{ob}). 
If $0 < u < \frac{1}{\sqrt{(q+t)(1-t)}}$, then, for $k \geq 0$, we have 
\begin{align}\label{G(t)f}
G(t)\big( \e^{-(q+1)\tau}\big((q+t)(1-t)\big)^{-\frac{k}{2}}\I_{k}\big(2\sqrt{(q+t)(1-t)}\tau\big)\big)(u)=u^{k-1}. 
\end{align}

\section{A generalized path counting formula}
In this section, we give a generalization of the path counting formula obtained by T. Kousaka (\cite{K2017}). 
First of all, we introduce several symbols. 
We take a vertex $x_{0}$ and $e \in E_{x_{0}}$. 
We denote by $\mathcal{C}_{x_{0}}$ the set of closed paths starting at $x_{0}$ and by $\mathcal{C}_{x_{0}}^{notail}$ the set of closed paths starting at $x_{0}$ which has no tail.  
For a complex variable $t$, we define $C_{m}(t, x_{0})$, $N_{m}(t, x_{0}, e)$ as follows. 
\begin{align*}
C_{m}(t, x_{0})=\sum_{C} t^{\cbc(C)}, \\
N_{m}(t, x_{0}, e)=\sum_{C} t^{\cbc(C)}. 
\end{align*}
Here, $C$ runs through $\mathcal{C}_{x_{0}}$ such that $\ell(C)=m$ in the first equality and $C$ runs through $\mathcal{C}_{x_{0}}^{notail}$ such that $e_{1}(C)=e, \ell(C)=m$ in the second equality. 
For $f \in \ell^{2}(V\!X)$ and $x \in V\!X$, we define $C_{m}(t)$ by 
\begin{align*}
C_{m}(t)f(x)=\sum_{C \in \mathcal{B}_{x}, \ell(C)=m}t^{\bc(C)}f(t(C)). 
\end{align*}
Here, we denote by $\mathcal{B}_{x}$ the set of paths starting at $x$. 
We define $C_{m}(t)(x_{0}, e)$ as follows. 
\begin{align*}
C_{m}(t)(x_{0}, e)=\sum_{C}t^{\bc(C)}. 
\end{align*}
Here, $C$ runs through $\mathcal{C}_{x_{0}}$ such that $e_{1}(C)=e, \ell(C)=m$ in the above equality.  
Moreover, we define $C_{m}(t)(x_{0}, \cdot, \bar{e})$, $N_{m}(t, x_{0}, \cdot, \bar{e})$ and $C_{m}(t)(x_{0}, e, \bar{e})$ as follows. 
\begin{align*}
C_{m}(t)( x_{0}, \cdot, \bar{e})=\sum_{C}t^{\bc(C)}, \\
N_{m}(t, x_{0}, \cdot, \bar{e})=\sum_{C}t^{\bc(C)}, \\
C_{m}(t)(x_{0}, e, \bar{e})=\sum_{C}t^{\bc(C)}. 
\end{align*}
Here, $C$ runs through $\mathcal{C}_{x_{0}}$ such that $e_{m}(C)=\bar{e}, \ell(C)=m$ in the first equality, $C$ runs through $\mathcal{C}_{x_{0}}^{notail}$ such that $e_{m}(C)=\bar{e}, \ell(C)=m$ in the second equality and $C$ runs through $\mathcal{C}_{x_{0}}$ such that $e_{1}(C)=e, e_{m}(C)=\bar{e}, \ell(C)=m$ in the third equality. 
We denote by $\mathcal{B}(\ell^{2}(V\!X))$ the set of bounded operators on $\ell^{2}(V\!X)$. 
We remark that $C_{m}(t)$ is in $\mathcal{B}(\ell^{2}(V\!X))$ for each $t$. 
For $B \in \mathcal{B}(\ell^{2}(V\!X))$ and $x_{1}, x_{2} \in V\!X$, we define $B(x_{1}, x_{2})$ as follows. 
\begin{align*}
B(x_{1}, x_{2})=B\delta_{x_{1}}(x_{2}). 
\end{align*}
Here, the symbol $\delta_{x_{0}}$ stands for the Kronecker delta. 
We define the following formal power series. 
\begin{align*}
C^{\cbc}(t, x_{0}: u)=\sum_{m=1}^{\infty}C_{m}(t, x_{0})u^{m}, \\
C(t, x_{0}: u)=\sum_{m=1}^{\infty}C_{m}(t)(x_{0}, x_{0})u^m, \\
N(t, x_{0}: u)=\sum_{m=1}^{\infty}N_{m}(t, x_{0})u^m.
\end{align*}
Here, we denote $\sum_{e \in E_{x_{0}}}N_{m}(t, x_{0}, e)$ by $N_{m}(t, x_{0})$. 
In addition to this, for a vertex $x \in V\!X$, we denote $\degg(x)C(t, x: u)-\sum_{e \in E_{x}}C(t, t(e): u)$ by $\Delta_{X}C(t, \cdot : u)(x)$ by regarding as an element of $\ell^{2}(V\!X)$ formally and for $m \geq 1$, $x \in V\!X$, we define $R_{m}(t)(x)$ as follows. 
\begin{align*}
R_{m}(t)(x)=\sum_{j=1}^{\lceil \frac{m}{2} \rceil -1}\sum_{i=1}^{j}(1-t)^{2(j-i)}(1-t^2)^{i-1}[\Delta_{X}C_{m-2j}(t)](x, x). 
\end{align*}

First of all, we prove the following proposition. 
\begin{proposition}\label{fNC}
For a vertex $x_{0}$, we have the following equality: 
\begin{align*}
\big(1-(1-t)^{2}u^{2}\big)N(t, x_{0}: u)
&=\big(1-(\degg(x_{0})-(1-t^{2}))u^{2}\big)C(t, x_{0}: u)\\
&\quad-\degg(x_{0})tu^{2}+\frac{u^{2}}{1-(1-t^{2})u^{2}}\Delta_{X}C(t, \cdot: u)(x_{0}). 
\end{align*}
Moreover, for $m \geq 3$, we have 
	\begin{align*}
	N_{m}(t, x_{0})&=C_{m}(t)(x_{0}, x_{0})
	-(\degg(x_{0})-2(1-t))
	\sum_{j=1}^{\lceil \frac{m}{2} \rceil -1}(1-t)^{2(j-1)}C_{m-2j}(t)(x_{0}, x_{0})\\
	&\quad+R_{m}(t)(x_{0})-\delta_{2\Z}(m)(1-t)^{m-2}t\degg(x_{0}). 
	\end{align*}
Here, we denote the ceiling function by $\lceil \cdot \rceil$. 
\end{proposition}
\begin{proof}
First of all, we prove the first identity. 
For $m \geq 3$, $x_{0} \in V\!X$ and $e \in E_{x_{0}}$, we have 
\begin{align*}
C_{m}(t)(x_{0}, e, \bar{e})
&=\sum_{C=(e, e_{2}, \dots, e_{m-1}, \bar{e}), e_{2}=\bar{e}, e_{m-1} \neq e}t^{\bc(C)}
+\sum_{C=(e, e_{2}, \dots, e_{m-1}, \bar{e}), e_{2}=\bar{e}, e_{m-1} = e}t^{\bc(c)}\\
&\quad+\sum_{C=(e, e_{2}, \dots, e_{m-1}, \bar{e}), e_{2} \neq \bar{e}, e_{m-1} \neq e}t^{\bc(C)}
+\sum_{C=(e, e_{2}, \dots, e_{m-1}, \bar{e}), e_{2} \neq \bar{e}, e_{m-1} = e}t^{\bc(C)}. 
\end{align*}
Then, we have 
\begin{align*}
&C_{m}(t)(x_{0}, e, \bar{e})\\
&=t\big(C_{m-2}(t)(t(e), \bar{e})-C_{m-2}(t)(t(e), \bar{e}, e)\big)+t^{2}C_{m-2}(t)(t(e), \bar{e}, e)\\
&\quad+\big( C_{m-2}(t)(t(e))-C_{m-2}(t)(t(e), \bar{e})-C_{m-2}(t)(t(e), \cdot, e)+C_{m-2}(t)(t(e), \bar{e}, e) \big)\\
&\quad+t\big( C_{m-2}(t)(t(e), \cdot, e)-C_{m-2}(t)(t(e), \bar{e}, e) \big)\\
&=C_{m-2}(t)(t(e), t(e))+(t-1)C_{m-2}(t)(t(e), \bar{e})\\
&\quad+(t-1)C_{m-2}(t)(t(e), \cdot, e)+(t-1)^{2}C_{m-2}(t)(t(e), \bar{e}, e). 
\end{align*}
By this, we have 
\begin{align*}
&C_{m}(t)(x_{0}, e, \bar{e})\\
&=C_{m-2}(t)(t(e), t(e))
+(t-1)\big(C_{m-2}(t)(t(e), \bar{e})-N_{m-2}(t, t(e), \bar{e}) \big)\\
&\quad+(t-1)\big(C_{m-2}(t)(t(e), \cdot, e)-N_{m-2}(t, t(e), \cdot, e) \big)\\
&\quad+(t-1)\big(N_{m-2}(t, t(e), \bar{e})+N_{m-2}(t, t(e), \cdot, e) \big)
+(t-1)^{2}C_{m-2}(t)(t(e), \bar{e}, e)\\
&=C_{m-2}(t)(t(e), t(e))+2(t-1)N_{m-2}(t, t(e), \bar{e})+(t^2-1)C_{m-2}(t)(t(e), \bar{e}, e). 
\end{align*}
Therefore, we get 
\begin{align}\label{f1}
C_{m}(t)(x_{0}, e, \bar{e})
&=C_{m-2}(t)(t(e), t(e))+2(t-1)N_{m-2}(t, t(e), \bar{e}) \nonumber \\
&\quad+(t^2-1)C_{m-2}(t)(t(e), \bar{e}, e). 
\end{align}
In the case that $m \geq 5$,  by (\ref{f1}), we have 
\begin{align*}
&C_{m}(t)(x_{0}, e, \bar{e})\\
&=C_{m-2}(t)(t(e), t(e))+2(t-1)N_{m-2}(t, t(e), \bar{e})\\
   &\quad+(t^2-1)\bigg\{ C_{m-4}(t)(x_{0}, x_{0})+2(t-1)N_{m-4}(t, x_{0}, e)+(t^{2}-1)C_{m-4}(t)(x_{0}, e, \bar{e})\bigg\}. 
\end{align*}
Then, we have 
\begin{align}\label{f2}
&C_{m}(t)(x_{0}, x_{0})-N_{m}(t, x_{0})\nonumber \\
&=\big(\degg(x_{0})C_{m-2}(t)(x_{0}, x_{0})-\Delta_{X}C_{m-2}(t)(x_{0}, x_{0})\big)\nonumber \\
	&\quad+2(t-1)N_{m-2}(t, x_{0})+(t^{2}-1)\degg(x_{0})C_{m-4}(t)(x_{0}, x_{0})+(t^{2}-1)2(t-1)N_{m-4}(t, x_{0})\nonumber \\
	&\quad+(t^{2}-1)^{2}\big( C_{m-4}(t)(x_{0}, x_{0})-N_{m-4}(t, x_{0}) \big). 
\end{align}
Hence, for $m \geq 1$, we have the following. 
\begin{align*}
&N_{m+4}(t, x_{0})-2(1-t)N_{m+2}(t, x_{0})+(1-t^{2})(1-t)^{2}N_{m}(t, x_{0})\\
&=C_{m+4}(t)(x_{0}, x_{0})-\degg(x_{0})C_{m+2}(t)(x_{0}, x_{0})-(1-t^{2})^{2}C_{m}(t)(x_{0}, x_{0})\\
&\quad+\degg(x_{0})(1-t^{2})C_{m}(t)(x_{0}, x_{0})+\Delta_{X}C_{m+2}(t)(x_{0}, x_{0}). 
\end{align*}
It is easy to check that this implies the following desired identity: 
\begin{align*}
\big(1-(1-t)^{2}u^{2}\big)N(t, x_{0}: u)
&=\big(1-(\degg(x_{0})-(1-t^{2}))u^{2}\big)C(t, x_{0}: u)\\
&\quad-\degg(x_{0})tu^{2}+\frac{u^{2}}{1-(1-t^{2})u^{2}}\Delta_{X}C(t, \cdot: u)(x_{0}). 
\end{align*}
Next, we prove the second identity. 
For $m=3$, by (\ref{f1}), we have 
\begin{align*}
&N_{3}(t, x_{0})-(1-t)^{2}N_{1}(t, x_{0})\\
&=C_{3}(t)(x_{0}, x_{0})-\big( \degg(x_{0})-(1-t^{2}) \big)C_{1}(t)(x_{0}, x_{0})+\Delta_{X}C_{1}(t)(x_{0}, x_{0}). 
\end{align*}
Then, we have
\begin{align*}
N_{3}(t, x_{0})=C_{3}(t)(x_{0}, x_{0})-\big( \degg(x_{0})-2(1-t) \big)C_{1}(t)(x_{0}, x_{0})+\Delta_{X}C_{1}(t)(x_{0}, x_{0}). 
\end{align*}
For $m=4$, by (\ref{f1}), we have 
\begin{align*}
N_{4}(t, x_{0})
&=C_{4}(t)(x_{0}, x_{0})-(\degg(x_{0})-(1-t^{2}))C_{2}(t)(x_{0}, x_{0})\\
	&\quad+\Delta_{X}C_{2}(t)(x_{0}, x_{0})+(1-t)^{2}N_{2}(t, x_{0})\\
&=C_{4}(t)(x_{0}, x_{0})-(\degg(x_{0})-(1-t^{2}))C_{2}(t)(x_{0}, x_{0})\\
	 &\quad+\Delta_{X}C_{2}(t)(x_{0}, x_{0})+(1-t)^{2}(C_{2}(t)(x_{0}, x_{0})-t\degg(x_{0})) \\
&=C_{4}(t)(x_{0}, x_{0})-(\degg(x_{0})-2(1-t))C_{2}(t)(x_{0}, x_{0})\\
	 &\quad+\Delta_{X}C_{2}(t)(x_{0}, x_{0})-(1-t)^{2}t\degg(x_{0}). 
\end{align*}
Therefore, the second identity holds for $m=3, 4$. 
In the case $m \geq 5$, the identity (\ref{f2}) is equivalent to the following identity. 
\begin{align}\label{f3}
&N_{m}(t, x_{0})-(1-t)^{2}N_{m-2}(t, x_{0})-(1-t^{2})\big(N_{m-2}(t, x_{0})-(1-t)^{2}N_{m-4}(t, x_{0}) \big)\nonumber \\
&=C_{m}(t)(x_{0}, x_{0})-(1-t^{2})C_{m-2}(t)(x_{0}, x_{0})\nonumber \\
	&\quad-\big( \degg(x_{0}-(1-t^{2}) \big)\big(C_{m-2}(t)(x_{0}, x_{0})-(1-t^{2})C_{m-4}(t)(x_{0}, x_{0}) \big) \nonumber \\
	&\quad+\Delta_{X}C_{m-2}(t)(x_{0}, x_{0}). 
\end{align}
By summing the both sides of (\ref{f3}), we have the desired identity. 
\end{proof}

Next, we prove the following theorem that is our goal in this section. 
\begin{theorem}\label{cbc}
For a vertex $x_{0}$, we have the following identity: 
\begin{align*}
&\big(1-(1-t^{2})u^{2}\big)\big(1-(1-t)^{2}u^{2}\big)C^{\cbc}(t, x_{0}: u)\\
&=\big\{1-(1-t)(\degg(x_{0})-\Delta_{X}+2t)u^{2}+(1-t^{2})(1-t)\big(\degg(x_{0})-(1-t)\big)u^{4} \big\}C(t, \cdot: u)(x_{0})\\
	&\quad-\big(1-(1-t^{2})\big)t\degg(x_{0})(1-t)u^{2}. 
\end{align*}
Moreover, for $m \geq 3$, we have the following identity: 
\begin{align*}
C_{m}(t, x_{0})&=C_{m}(t)(x_{0}, x_{0})
	-\frac{\degg(x_{0})-2(1-t)}{1-t}\sum_{j=1}^{\lceil \frac{m}{2} \rceil-1}(1-t)^{2j}C_{m-2j}(t)(x_{0}, x_{0})\\
	&\quad+(1-t)R_{m}(t)(x_{0})-\delta_{2\Z}(m)(1-t)^{m-1}t\degg(x_{0}). 
\end{align*}
\end{theorem}
\begin{proof}
For $m \geq 1$, $C_{m}(t, x_{0})-N_{m}(t, x_{0})$ ( resp.~$C_{m}(t)(x_{0}, x_{0})-N_{m}(t, x_{0})$ )  represents the number of closed paths with weight $t^{\cbc(\cdot)}$ ( resp.~$t^{\bc(\cdot)}$ ) of length $m$ starting at $x_{0}$, which have no tail. 
Hence, we have 
\begin{align*}
C_{m}(t, x_{0})-N_{m}(t, x_{0})=t\big(C_{m}(t)(x_{0}, x_{0})-N_{m}(t, x_{0})\big). 
\end{align*}
Then, we have 
\begin{align*}
C^{\cbc}(t, x_{0}: u)=tC(t, c_{0}: u)+(1-t)N(t, x_{0}: u). 
\end{align*}
By this and Proposition \ref{fNC}, we have 
\begin{align*}
&\big(1-(1-t)^{2}u^{2}\big)C^{\cbc}(t, x_{0}: u)\\
&=\bigg( (1-t)(1-\big(\degg(x_{0})-(1-t^{2})\big)u^{2})+t(1-(1-t)^{2}u^{2})\bigg)C(t, x_{0}:u)\\
	&\quad -(1-t)t\degg(x_{0})u^{2}+\frac{(1-t)u^{2}}{1-(1-t^{2})u^{2}}\Delta_{X}C(t, \cdot :u)(x_{0}). 
\end{align*}
By simple calculation, this implies the first equality. 
Next, we verify the second equality. 
By Proposition \ref{fNC}, for $m \geq 3$, we have 
\begin{align*}
C_{m}(t, x_{0})
&=t\big(C_{m}(t)(x_{0}, x_{0})-N_{m}(t, x_{0}) \big)\\
&\hspace{0.5cm}+\big(N_{m}(t, x_{0})-C_{m}(t)(x_{0}, x_{0})\big)+C_{m}(t)(x_{0}, x_{0}) \\
&=C_{m}(t)(x_{0}, x_{0})-(1-t)\big(C_{m}(t)(x_{0}, x_{0})-N_{m}(t, x_{0}) \big)\\
&=C_{m}(t)(x_{0}, x_{0}) \\
&\quad-(1-t)(\degg(x_{0})-2(1-t))\sum_{j=1}^{\lceil \frac{m}{2} \rceil -1}(1-t)^{2(j-1)}C_{m-2j}(t)(x_{0}, x_{0})\\
&\quad+(1-t)R_{m}(t)(x_{0})-\delta_{2\Z}(m)(1-t)^{m-1}t \degg(x_{0}). 
\end{align*}
\end{proof}

In the end of this section, we note generating functions which we defined in this section. 
We define the generating function of $R_{m}(t)(x_{0})$ as follows. 
\begin{align*}
R(t, x_{0}: u)=\sum_{m=1}^{\infty}R_{m}(t)(x_{0})u^{m}. 
\end{align*}
It is straightforward to check that the following holds. 
If $\left| t \right| <1$, $t \neq 0$ and $\left| u \right| < \frac{1}{2}$, then, we have 
\begin{align*}
R(t, x_{0}: u)=\frac{u^{2}}{\big(1-(1-t)^{2}u^{2}\big)\big(1-(1-t^{2})u^{2}\big)}\Delta_{X}C(t, \cdot: u)(x_{0}). 
\end{align*}
Therefore, all generating functions which we defined in this section are expressed by $C(t, x_{0}: u)$. 
We note that the above formula holds for $t=0$. 

\section{A generalized Bartholdi zeta function formula for simple graphs with bounded degree}
In this section, we introduce a Bartholdi zeta function for a graph with bounded degree. 
This zeta function is a generalization of the Bartholdi zeta function from a finite graph to a graph with bounded degree (\cite{Bartholdi1999}). 

First of all, we define a Bartholdi zeta function for a graph with bounded degree. 
For a graph $X$ with bounded degree, a vertex $x_{0}$ and complex variables $t$, $u$, we define a {\it Barholdi zeta function} as follows. 
\begin{align*}
Z_{X}(u, t, x_{0})=\exp \bigg( \sum_{C \in \mathcal{C}_{x_{0}}}\frac{1}{\ell(C)}t^{\cbc(C)}u^{\ell(C)} \bigg). 
\end{align*}
This is a natural generalization of the Ihara zeta function for a graph with bounded degree which was  introduced in \cite{CJK1} in the spirit of L.~Bartholdi although he introduced by using the Euler product expression. 
Before we give an Ihara type formula for this zeta function, we define several operators and give several properties of $C_{m}(t)$.  
We define $I(t)$, $Q_{X}$ and $Q_{X}(t)$ as follows. 
\begin{align*}
I(t)&=(1-t)I, \\
Q_{X}&=D_{X}-I, \\
Q_{X}(t)&=D_{X}-I(t). 
\end{align*}
Then, we have the following proposition. 
\begin{proposition}\label{Cm}
For $m \geq 2$, we have 
\begin{align*}
C_{m}(t)=\left \{ \begin{array}{ll}
			C_{1}(t)^{2}-(1-t)(Q_{X}+I)		& \text{if $m=2$}, \\
			C_{m-1}(t)C_{1}(t)-(1-t)C_{m-2}(t)Q_{X}(t)	& \text{if $m\geq 3$}.\\
			\end{array} \right. 
\end{align*}
\end{proposition}
We give the proof of this proposition although this proposition was proved in \cite{IS2009} because our proof is a little different from \cite{IS2009}. 
\begin{proof}
It is enough to show that for $x_{0}, x \in V\!X$, 
\begin{align*}
C_{m}(t)(x_{0}, x)=\left \{ \begin{array}{ll}
			\big( C_{1}(t)^{2}-(1-t)(Q_{X}+I)\big)(x_{0}, x)	& \text{if $m=2$},\\
			\big( C_{m-1}(t)C_{1}(t)-(1-t)C_{m-2}(t)Q_{X}(t)\big)(x_{0}, x)	& \text{if $m\geq 3$}. \\
			\end{array} \right. 
\end{align*}
For $m=2$, it is obvious. 
In the case $m \geq 3$, for $x_{0}, x \in V\!X$, we have 
\begin{align*}
C_{m-1}(t)C_{1}(t)(x_{0}, x)
=C_{m-1}(t)C_{1}(t)\delta_{x_{0}}(x)
=\sum_{C \in \mathcal{B}_{x}, \ell(C)=m-1}\sum_{e \in E_{t(C)}^{x_{0}}}t^{\bc(C)}. 
\end{align*}
By considering whether a path has backtracking at the last step and comparing $C_{m-1}(t)C_{1}(t)(x_{0}, x)$ to $C_{m}(t)(x_{0}, x)$, we have 
\begin{align*}
&C_{m-1}(t)C_{1}(t)(x_{0}, x)
-C_{m-2}(t)(x_{0}, x)t-C_{m-2}(t)(x_{0}, x)(\degg(x_{0})-1)\\
&\quad+C_{m-2}(t)(x_{0}, x)t^{2}+C_{m-2}(t)(x_{0}, x)(\degg(x_{0})-1)t\\
&=C_{m}(t)(x_{0}, x). 
\end{align*}
Therefore, we have 
\begin{align*}
C_{m}(t)&=C_{m-1}(t)C_{1}(t)-C_{m-2}(t)t-C_{m-2}(t)Q_{X}+C_{m-2}(t)t^{2}+C_{m-2}(t)Q_{X}t\\
&=C_{m-1}(t)C_{1}(t)-(1-t)C_{m-2}(t)Q_{X}(t). 
\end{align*}
\end{proof}

For a complex valuable $t$, we define $\alpha(t)$ by 
\begin{align*}
\alpha(t)=\frac{M+\sqrt{M^{2}+4(\left| t \right|+1)M}}{2}. 
\end{align*}
Here, we denote the maximum of all degrees of $X$ by $M$. 
Then, we have the following Lemma. 
\begin{lemma}\label{at}
For $\left| t \right| <1$, then for $m \geq 0$, we have 
\begin{align*}
\lVert C_{m}(t) \rVert  \leq \alpha(t)^{m}. 
\end{align*}
\end{lemma}
\begin{proof}
We prove this by induction on $m$. 
For $m=0, 1$, there is nothing to do. 
We suppose that our assertion holds for $m-1$. Then, we have 
\begin{align*}
\lVert C_{m}(t) \rVert 
&=\lVert C_{m-1}(t)C_{1}(t)-(1-t)C_{m-2}(t)Q_{X}(t) \rVert \\
& \leq M \alpha(t)^{m-1}+(1+\left| t \right|)\alpha(t)^{m-2}(M-1+\left| t \right|)\\
&=\alpha(t)^{m-2}\big\{ \alpha(t)M+(1+\left| t \right|)(M-1+\left| t \right|) \big\} \\
&=\alpha(t)^{m-2}(\alpha(t)^{2}+\left| t \right|^{2}-1)< \alpha(t)^{m}. 
\end{align*}
\end{proof}
By Proposition \ref{Cm} and Lemma \ref{at}, we have the following proposition. 
\begin{proposition}\label{fC}
For $\left| u \right| < \frac{1}{\alpha(t)}$, $\left| t \right| <1$, we have 
\begin{align*}
	\bigg(\sum_{m=0}^{\infty}C_{m}(t)u^m\bigg)\bigg(I-uA_{X}+(1-t)Q_{X}(t)u^{2}\bigg)
	=(1-(1-t)^2 u^{2})I, \\
	\bigg(\sum_{m=0}^{\infty}
	\big(\sum_{j=0}^{\lfloor \frac{m}{2} \rfloor}C_{m-2j}(t)(1-t)^{2j}\big)u^{m}\bigg)
	\bigg(I-uA_{X}+(1-t)Q_{X}(t)u^{2} \bigg)=I. 
\end{align*}
\end{proposition}

Next, for $m \leq 0$, $t$, $x \in V\!X$ and $f \in \ell^{2}(V\!X)$, we define an operator $R_{m}(t)$ by 
\begin{align*}
R_{m}(t)f(x)=R_{m}(t)(x)f(x). 
\end{align*}
Then, we define an operator $C_{m}^{\cbc}(t)$ like the operator $N_{X, m}$ introduced in \cite{K2017} by 
\begin{align*}
C_{m}^{\cbc}(t)=\left \{ \begin{array}{lll}
			C_{m}(t)	& \text{if $m=0, 1$}, \\
			tC_{2}(t)   & \text{if $m=2$}, \\
			C_{m}(t)-\frac{Q_{X}(t)-I(t)}{1-t}
			\sum_{j=1}^{\lceil \frac{m}{2} \rceil -1}(1-t)^{2j}C_{m-2j}(t)\\
			+(1-t)R_{m}(t)-\delta_{2\Z}(m)(1-t)^{m-1}tD_{X}	& \text{if $m \geq 3$}.\\
			\end{array} \right. 
\end{align*}
We remark that this operator is also a bounded operator by our assumption. 
We also remark that $C_{m}^{\cbc}(t)(x_{0}, x_{0})=C_{m}(t, x_{0})$ and the following identity holds by Theorem \ref{cbc}. 
\begin{align*}
Z_{X}(u, t, x_{0})=\exp\bigg( \sum_{m=1}^{\infty}\frac{1}{m}C_{m}(t, x_{0})u^{m} \bigg). 
\end{align*}
Therefore, for $x_{0}, x \in V\!X$, we define $Z_{X}(u, t, x_{0}, x)$ as follows. 
\begin{align*}
Z_{X}(u, t, x_{0}, x)=\exp\bigg( \sum_{m=1}^{\infty}\frac{C_{m}^{\cbc}(t)(x_{0}, x)}{m}u^{m}\bigg). 
\end{align*}
Moreover, we define $f(z)$ as follows. 
\begin{align*}
f(z)=zA_{X}-z^{2}(1-t)Q_{X}(t). 
\end{align*}
Then, we have the following Proposition (\cite{K2017}). 
\begin{proposition}\label{f(u)}
For $\left| u \right| < \frac{1}{\alpha(t)}$, we have 
\begin{align*}
f'(u)(I-f(u))^{-1}
&=-\frac{d}{du}\log(I-f(u))\\
	&\quad+(1-t)u^2\sum_{n=1}^{\infty}\frac{1}{n}\sum_{j=1}^{n-1}jf(u)^{n-1-j}(A_{X}Q_{X}(t)-Q_{X}(t)A_{X})f(u)^{j-1}. 
\end{align*}
\end{proposition}
Under the above preparation, we give the following theorem. 
\begin{theorem}\label{Ihara formula}
For $\left| u \right| < \frac{1}{\alpha(t)}, \left| t \right| <1$ and $x_{0}, x \in V\!X$, we have 
\begin{align*}
&Z_{X}(u, t, x_{0}, x)\\&=(1-(1-t)^2 u^2)^{-\frac{\degg(x_{0})-2}{2}\delta_{x_{0}}(x)}
		\exp\bigg(-\left[\log(I-u(D_{X}-\Delta_{X})+(1-t)Q_{X}(t)u^{2})\right](x_{0}, x)\bigg)\\&
			\quad \times\exp\bigg( \int_{0}^{u}(1-t)z^{2}\sum_{n=2}^{\infty}\frac{1}{n}
					\sum_{j=1}^{n-1}j\left[f(z)^{n-1-j}
					(A_{X}D_{X}-D_{X}A_{X})
					f(z)^{j-1}\right](x_{0}, x)dz \bigg)\\
			&\quad \times \exp\bigg( \frac{\big[tD_{X}-C_{2}(t)\big](x_{0}, x)}{2} (1-t)u^{2}\bigg)
				\exp\bigg(\sum_{m=3}^{\infty}\frac{(1-t)R_{m}(t)(x_{0}, x)}{m}u^{m}\bigg). 
\end{align*}
\end{theorem}
\begin{proof}
We consider the following power series that converges in $\left| u \right| < \frac{1}{\alpha(t)}, \left| t \right| <1$. 
\begin{align*}
\sum_{m=0}^{\infty}C_{m}^{\cbc}(t)u^m. 
\end{align*}
By the definition of $C_{m}^{\cbc}(t)$, we have 
\begin{align*}
\sum_{m=0}^{\infty}C_{m}^{\cbc}(t)u^{m}
&=\frac{Q_{X}(t)}{1-t}\sum_{m=0}^{\infty}C_{m}(t)u^{m}
	-\frac{Q_{X}(t)-I(t)}{1-t}\sum_{m=0}^{\infty}\sum_{j=0}^{\lfloor \frac{m}{2} \rfloor}C_{m-2j}(t)(1-t)^{2j}u^{m}\\
	&\quad+\big( Q_{X}(t)-I(t) \big)\frac{(1-t)u^{2}}{1-(1-t)^{2}u^{2}}+(1-t)\sum_{m=3}^{\infty}R_{m}(t)u^{m}\\
	&\quad+C_{2}(t)(t-1)u^{2}-\frac{t(1-t)^{3}u^{4}}{1-(1-t)^{2}u^{2}}D_{X}.  
\end{align*}
Here, we denote the floor function by $\lfloor \cdot \rfloor$. 
In the right hand side of the above equation, the following equality holds. 
\begin{align*}
&\big( Q_{X}(t)-I(t) \big)\frac{(1-t)u^{2}}{1-(1-t)^{2}u^{2}}
-\frac{t(1-t)^{3}u^{4}}{1-(1-t)^{2}u^{2} }D_{X}\\
&=t(1-t)u^{2}D_{X}+\frac{(1-t)^{2}u^{2}}{1-(1-t)^{2}u^{2}}(Q_{X}-I). 
\end{align*}
Hence, we have  
\begin{align*}
\sum_{m=0}^{\infty}C_{m}^{\cbc}(t)u^{m}
&=\frac{Q_{X}(t)}{1-t}\sum_{m=0}^{\infty}C_{m}(t)u^{m}
    -\frac{Q_{X}(t)-I(t)}{1-t}\sum_{m=0}^{\infty}\sum_{j=0}^{\lfloor \frac{m}{2} \rfloor}C_{m-2j}(t)(1-t)^{2j}u^{m}\\
	&\quad+(Q_{X}-I)\frac{(1-t)^{2}u^{2}}{1-(1-t)^{2}u^{2}}+(1-t)u^{2}\big(tD_{X}-C_{2}(t)\big)\\
	&\quad+(1-t)\sum_{m=3}^{\infty}R_{m}(t)u^{m}. 
\end{align*}
By this, we have
\begin{align*}
\sum_{m=1}^{\infty}C_{m}^{\cbc}(t)u^{m}
&=\frac{Q_{X}(t)}{1-t}\sum_{m=0}^{\infty}C_{m}(t)u^{m}-I
    -\frac{Q_{X}(t)-I(t)}{1-t}\sum_{m=0}^{\infty}\sum_{j=0}^{\lfloor \frac{m}{2} \rfloor}C_{m-2j}(t)(1-t)^{2j}u^{m}\\
	&\quad+(Q_{X}-I)\frac{(1-t)^{2}u^{2}}{1-(1-t)^{2}u^{2}}+(1-t)u^{2}\big(tD_{X}-C_{2}(t)\big)\\
	&\quad+(1-t)\sum_{m=3}^{\infty}R_{m}(t)u^{m}. 
\end{align*}
By Proposition \ref{fC}, we have 
\begin{align*}
&\frac{Q_{X}(t)}{1-t}\sum_{m=0}^{\infty}C_{m}(t)u^{m}-I
    -\frac{Q_{X}(t)-I(t)}{1-t}\sum_{m=0}^{\infty}\sum_{j=0}^{\lfloor \frac{m}{2} \rfloor}C_{m-2j}(t)(1-t)^{2j}u^{m}\\
&=\frac{Q_{X}(t)}{1-t}\big(1-(1-t)^{2}u^{2}\big)(I-f(u))^{-1}-I-\frac{Q_{X}(t)-I(t)}{1-t}(I-f(u))^{-1}\\
&=u f'(u)(I-f(u))^{-1}. 
\end{align*}
By Proposition \ref{f(u)}, we have 
\begin{align*}
&\frac{Q_{X}(t)}{1-t}\sum_{m=0}^{\infty}C_{m}(t)u^{m}-I
    -\frac{Q_{X}(t)-I(t)}{1-t}\sum_{m=0}^{\infty}\sum_{j=0}^{\lfloor \frac{m}{2} \rfloor}C_{m-2j}(t)(1-t)^{2j}u^{m}\\
&=-u\frac{d}{du}\log(I-f(u))\\
	&\quad+(1-t)u^{3}\sum_{n=1}^{\infty}\frac{1}{n}\sum_{j=1}^{n-1}jf(u)^{n-1-j}
	\big(A_{X}Q_{X}(t)-Q_{X}(t)A_{X}\big)f(u)^{j-1}. 
\end{align*}
Therefore, for $x_{0}, x \in V\!X$, we have 
\begin{align*}
&u\frac{d}{du}\sum_{m=1}^{\infty}\frac{C_{m}^{\cbc}(t)(x_{0}, x)}{m}u^{m}\\
&=-\bigg[u\frac{d}{du}\log(I-f(u))\bigg](x_{0}, x)\\
	&\quad+(1-t)u^{3}\sum_{m=1}^{\infty}\frac{1}{n}\sum_{j=1}^{n-1}j 
	\big[ f(u)^{n-1-j}\big(A_{X}Q_{X}(t)-Q_{X}(t)A_{X}\big)f(u)^{j-1}\big](x_{0}, x)\\
	&\quad-\frac{\degg(x_{0})-2}{2}\delta_{x_{0}}(x) u \frac{d}{du}\big[\log(1-(1-t)^{2}u^{2})\big]
		+\frac{\big[tD_{X}-C_{2}(t)\big](x_{0}, x)}{2} u \frac{d}{du}\big[(1-t)u^{2}\big]\\
	&\quad+u\frac{d}{du}\sum_{m=3}^{\infty}\frac{(1-t)R_{m}(t)(x_{0}, x)}{m}u^{m}. 
\end{align*}

Dividing by $u$ and integrating from $0$ to $u$, we have 
\begin{align*}
&\sum_{m=1}^{\infty}\frac{C_{m}^{\cbc}(t)(x_{0}, x)}{m}u^{m}\\
&=-\bigg[\log(I-f(u))\bigg](x_{0}, x_{0})\\
&\quad+\int_{0}^{u}(1-t)z^{2}\sum_{m=1}^{\infty}\frac{1}{n}
	\sum_{j=1}^{n-1}j \big[ f(z)^{n-1-j}\big(A_{X}Q_{X}(t)-Q_{X}(t)A_{X}\big)f(z)^{j-1}\big](x_{0}, x)dz\\
&\quad-\frac{\degg(x_{0})-2}{2}\delta_{x_{0}}(x)\log(1-(1-t)^{2}u^{2})
+\frac{\big[tD_{X}-C_{2}(t)\big](x_{0}, x)}{2} (1-t)u^{2}\\
&\quad+\sum_{m=3}^{\infty}\frac{(1-t)R_{m}(t)(x_{0}, x)}{m}u^{m}. 
\end{align*}
This implies the following identity 
\begin{align*}
&Z_{X}(u, t, x_{0}, x)\\&=(1-(1-t)^2 u^2)^{-\frac{\degg(x_{0})-2}{2}\delta_{x_{0}}(x)}\\
		&\quad \times 
			\exp\bigg(-\left[\log(I-u(D_{X}-\Delta_{X})+(1-t)Q_{X}(t)u^{2})\right](x_{0}, x)\bigg)\\&
			\quad \times\exp\bigg( \int_{0}^{u}(1-t)z^{2}\sum_{n=2}^{\infty}\frac{1}{n}
					\sum_{j=1}^{n-1}j\left[f(z)^{n-1-j}
					(A_{X}D_{X}-D_{X}A_{X})
					f(z)^{j-1}\right](x_{0}, x)dz \bigg)\\
			&\quad \times \exp\bigg( \frac{\big[tD_{X}-C_{2}(t)\big](x_{0}, x)}{2} (1-t)u^{2}\bigg)
				\exp\bigg(\sum_{m=3}^{\infty}\frac{(1-t)R_{m}(t)(x_{0}, x)}{m}u^{m}\bigg). 
\end{align*}
\end{proof}

If $X$ is a $(q+1)$-regular graph, we have 
\begin{align*}
I-(D_{X}-\Delta_{X})u+(1-t)Q_{X}(t)u^{2}=I-\big((q+1)I-\Delta_{X}\big)u+(1-t)(q+t)u^{2}I. 
\end{align*}
Since $\Delta_{X}$ is a self-adjoint bounded operator, there exists a unique spectral measure $E$ such that 
\begin{align*}
\Delta_{X}=\int_{\sigma(\Delta_{X})}\lambda dE(\lambda). 
\end{align*}
Here, we denote the spectrum of the Laplacian $\Delta_{X}$ by $\sigma( \Delta_{X})$. 
By Theorem \ref{Ihara formula} and the property of the spectral integral, we have 
\begin{align*}
Z_{X}(u, t, x_{0}, x)
&=(1-(1-t)^{2}u^{2})^{-\frac{q-1}{2}\delta_{x_{0}}(x)}\\
&\quad \times \exp\bigg( \int_{\sigma(\Delta_{X})}-\log(1-(q+1-\lambda)u+(1-t)(q+t)u^{2})d\mu_{x_{0}, x}(\lambda)\bigg)\\
&\quad \times \exp\bigg(\frac{\big[tD_{X}-C_{2}(t)\big](x_{0}, x)}{2}(1-t)u^{2}\bigg)\\
&\quad\times \exp\bigg( \sum_{m=3}^{\infty}\frac{(1-t)R_{m}(t)(x_{0}, x)}{m}u^{m}\bigg). 
\end{align*}
Here, we denote $d\langle E(\lambda)\delta_{x_{0}}, \delta_{x_{0}} \rangle$ by $d\mu_{x_{0}, x_{0}}(\lambda)$. 
Hence, we get the following corollary. 
\begin{corollary}\label{Ihara formula for a regular graph}
For $\left| u \right| < \frac{1}{\alpha(t)}$, $\left| t \right| <1$, we have 
\begin{align*}
Z_{X}(u, t, x_{0}, x)
&=(1-(1-t)^{2}u^{2})^{-\frac{q-1}{2}\delta_{x_{0}}(x)}\\
&\quad\times \exp\bigg( \int_{\sigma(\Delta_{X})}-\log(1-(q+1-\lambda)u+(1-t)(q+t)u^{2})d\mu_{x_{0}, x}(\lambda)\bigg)\\
&\quad \times \exp\bigg(\frac{\big[tD_{X}-C_{2}(t)\big](x_{0}, x)}{2}(1-t)u^{2}\bigg)\\
&\quad\times \exp\bigg( \sum_{m=3}^{\infty}\frac{(1-t)R_{m}(t)(x_{0}, x)}{m}u^{m}\bigg). 		
\end{align*}
\end{corollary}

Moreover, we discuss the case that $X$ is a finite $(q+1)$-regular graph. 
We introduce the notion of the local spectrum (\cite{Fiol}). 
For a vertex $x \in V\!X$, we denote $x$-local multiplicity of $\lambda_{i}$ by $\m_{x}(\lambda_{i})$. 
Here, the {\it$x$-local multiplicity} of $\lambda_{i}$ is the $xx$-entry of the primitive idempotent $E_{\lambda_{i}}$. Let $\{\mu_{0}=\lambda_{0}, \mu_{1}, \dots, \mu_{d_{x}}\}$ be the set of eigenvalues whose local multiplicities are positive. 
For each vertex $x \in V\!X$, we denote the $x$-local spectrum by $\sigma_{x}(X)$. 
Here, the {\it $x$-local spectrum} is $\sigma_{x}(X)=\{\lambda_{0}^{\m_{x}(\lambda_{0})}, \mu_{1}^{\m_{x}(\mu_{1})}, \dots, \mu_{d_{x}}^{\m_{x}(\mu_{d_{x}})}\}$. 

Then, we have the following corollary immediately by Corollary \ref{Ihara formula for a regular graph}. 
\begin{corollary}
For $\left| u \right| < \frac{1}{\alpha(t)}$, $\left| t \right| <1$, we have 
\begin{align*}
Z_{X}(t, u, x_{0}, x)
&=(1-(1-t)^{2}u^{2})^{-\frac{q-1}{2}\delta_{x_{0}}(x)}\\
	&\quad \prod_{\lambda \in \sigma_{x_{0}}(\Delta_{X})}
	\big(1-(q+1-\lambda)u+(1-t)(q+t)u^2\big)^{-m_{x_{0}}(\lambda)}\\
	&\quad \times \exp\bigg(\frac{\big[tD_{X}-C_{2}(t)\big](x_{0}, x)}{2}(1-t)u^{2}\bigg)\\
	&\quad \times\exp\bigg(\sum_{m=3}^{\infty}\frac{(1-t)R_{m}(t)(x_{0}, x)}{m}u^m\bigg). 
\end{align*}
\end{corollary}

\section{The Euler product expression}
In this section, we give the Euler product expression of the Bartholdi zeta function which is introduced in Section $4$. 
We have to introduce several terminologies to give the Euler product expression. 
We take a vertex $x_{0}$. 
A closed path $C$ starting at $x_{0}$ is {\it primitive} if there is no closed paths starting at $x_{0}$ whose length is shorter than $\ell(C)$ and of which the multiple is $C$. 
We denote by $\mathcal{PK}_{x_{0}}$ the set of primitive closed paths starting at $x_{0}$. 
Then, the following theorem holds. 
\begin{theorem}
For $\left| u \right| < \frac{1}{\alpha(t)}$, $\left| t \right| <1$, we have 
\begin{align*}
Z_{X}(t, u, x_{0})=\prod_{C \in \mathcal{PK}_{x_{0}}}(1-t^{\cbc(C)}u^{\ell(C)})^{-\frac{1}{\ell(C)}}. 
\end{align*}
\end{theorem}
\begin{proof}
For $\left| u \right| < \frac{1}{\alpha(t)}$, $\left| t \right| <1$ and $N \in \Z_{\geq 1}$, we have 
\begin{align*}
\log \prod_{C \in \mathcal{PK}_{x_{0}}, \ell(C) \leq N}(1-t^{\cbc(C)}u^{\ell(C)})^{-\frac{1}{\ell(C)}}
&= -\sum_{C \in \mathcal{PK}_{x_{0}}, \ell(C) \leq N} \frac{1}{\ell(C)}\log(1-t^{\cbc(C)}u^{\ell(C)})\\
&= \sum_{C \in \mathcal{PK}_{x_{0}}, \ell(C) \leq N}\sum_{m=1}^{\infty}\frac{1}{\ell(C)m}t^{\cbc(C)m}u^{\ell(C)m} \\
&= \sum_{C \in \mathcal{PK}_{x_{0}}, \ell(C) \leq N}\sum_{m=1}^{\infty}\frac{1}{\ell(C^{m})}t^{\cbc(C^{m})}u^{\ell(C^{m})}\\
&= \sum_{C} \frac{1}{\ell(C)}t^{\cbc(C)}u^{\ell(C)}. 
\end{align*}
Here, the last sum runs through the set of closed paths starting at $x_{0}$ such that the length of primitive paths is less than or equal to $N$. 
Therefore, for any $N$, we have, 
\begin{align*}
\exp\bigg(  \sum_{C} \frac{1}{\ell(C)}t^{\cbc(C)}u^{\ell(C)} \bigg)
=\prod_{C \in \mathcal{PK}_{x_{0}}, \ell(C) \leq N}(1-t^{\cbc(C)}u^{\ell(C)})^{-\frac{1}{\ell(C)}}
\end{align*}
Here, the sum runs through the same set as the above. 
Taking the limit of the both sides, we have 
\begin{align*}
Z_{X}(t, u, x_{0})=\prod_{C \in \mathcal{PK}_{x_{0}}}(1-t^{\cbc(C)}u^{\ell(C)})^{-\frac{1}{\ell(C)}}. 
\end{align*}
\end{proof}

\section{The heat kernels on regular graphs}
In this section, for a regular graph, we give a new expression of the heat kernels on regular graphs by using the modified Bessel function of the first kind. 
Let $X$ be a $(q+1)$-regular graph. 
We denote the heat kernel of $X$ by $K_{X}(\tau, x_{0}, x)$. 
For $j \in \Z_{\geq 0}$ and real variable $t$ which satisfies $\left| t \right| <1$, we define the symbol $d_{j}(t)$ as follows. 
\begin{align*}
d_{j}(t)=\left \{ \begin{array}{ll}
			1	& \text{if $j=0$}, \\
			-\frac{q-1+2t}{1-t}	& \text{if $j\geq 1$}.\\
			\end{array} \right. 
\end{align*}
Then, the following theorem holds. 
\begin{theorem}\label{heat kernel}
For $\tau \in \R_{\geq 0}$, $x \in V\!X$ and $\left| t \right| <1$, we have 
\begin{align*}
K_{X}(\tau, x_{0}, x)&=\sum_{n=0}^{\infty}C_{n}(t)(x_{0}, x)
			\sum_{j=0}^{\infty}d_{j}(t) \e^{-(q+1)\tau} \\
			&\quad \times (1-t)^{2j}\big((1-t)(q+t)\big)^{-\frac{n+2j}{2}}
			\I_{n+2j}\big(2\sqrt{(1-t)(q+t)}\tau\big). 
\end{align*}
\end{theorem}
\begin{proof}
We define $g(\tau, x)$ as follows. 
\begin{align*}
g(\tau, x)=\e^{(q+1)\tau}f(\tau, x). 
\end{align*}
Then, it turns out that the heat equation is equivalent to the following equation. 
\begin{align*}
\left \{ \begin{array}{ll}
	\frac{\partial g}{\partial \tau}(\tau, x)-C_{1}(t)g(\tau, \cdot)(x)=0 ,\\
	g(0, x)=\delta_{x_{0}}(x). 
	\end{array} \right. 
\end{align*}
Here, we remark that $C_{1}(t)$ is equal to $A_{X}$. 
It is sufficient to show that the following is the solution of the above equation. 
\begin{align*}
g(\tau, x)&=\sum_{n=0}^{\infty}C_{n}(t)(x_{0}, x)
			\sum_{j=0}^{\infty}d_{j}(t)\\
			&\quad \times (1-t)^{2j}\big((1-t)(q+t)\big)^{-\frac{n+2j}{2}}
			\I_{n+2j}\big(2\sqrt{(1-t)(q+t)}\tau\big).
\end{align*}
It is obvious that $g(0, x)=\delta_{x_{0}}(x)$. 
Therefore, it remains to check that $g(\tau, x)$ is bounded on $[0, T] \times V\!X$ for each $T$ and $g(\tau, x)$ satisfies the above equation indeed. 

First, we check that $g(\tau, x)$ is bounded. 
By Proposition \ref{Cm} and $(\ref{Bbound})$, we have 
\begin{align*}
\left| g(\tau, x) \right| \leq \sum_{n=0}^{\infty} \alpha(t)^{n}
	\sum_{j=0}^{\infty}\left| d_{j}(t) \right| \left| 1-t \right|^{2j} 
	\tau^{n+2j}\frac{\e^{2\sqrt{(q+t)(1-t)}\tau}}{(n+2j)!}. 
\end{align*}
We denote the maximum of $d_{j}(t)$ by $M_{t}$. 
Then, we have 
\begin{align*}
\left| g(\tau, x) \right| 
&\leq M_{t} \e^{2\sqrt{(q+t)(1-t)}\tau} 
\sum_{n=0}^{\infty}\alpha(t)^{n}\sum_{j=0}^{\infty}\frac{\tau^{n}(2\tau)^{2j}}{(n+2j)!}\\
&\leq M_{t} \e^{2\sqrt{(q+t)(1-t)}\tau}
\sum_{n=0}^{\infty}\frac{\big(\alpha(t)\tau\big)^{n}}{n!}
\sum_{j=0}^{\infty}\frac{(2\tau)^{2j}}{(2j)!}\\
&= M_{t} \e^{(2\sqrt{(q+t)(1-t)}+\alpha(t))\tau}\cosh (2\tau). 
\end{align*}
Therefore, $g(\tau, x)$ is bounded on $[0, T] \times V\!X$ for each $T$. 

Second, we check that $g(\tau, x)$ satisfies the equation. 
To prove this, we define $g(\tau)$  as follows. 
\begin{align*}
g(\tau)&=\sum_{n=0}^{\infty}C_{n}(t)
			\sum_{j=0}^{\infty}d_{j}(t)\\
			&\quad \times (1-t)^{2j}\big((1-t)(q+t)\big)^{-\frac{n+2j}{2}}
			\I_{n+2j}\big(2\sqrt{(1-t)(q+t)}\tau\big).
\end{align*}
Then, 
\begin{align*}
&\frac{\partial g}{\partial \tau}(\tau, x)-C_{1}(t)g(\tau, \cdot)(x)\\
&=\bigg\{ \frac{\partial g}{\partial \tau}(\tau)-C_{1}(t)g(\tau)\bigg\}(x_{0}, x). 
\end{align*}
Therefore, it is sufficient to check that 
\begin{align*}
\frac{\partial g}{\partial \tau}(\tau)-C_{1}(t)g(\tau)=0. 
\end{align*}
\begin{align*}
&\frac{\partial g}{\partial \tau}(\tau)-C_{1}(t)g(\tau)\\
&=\sum_{n=0}^{\infty}C_{n}(t)
	\sum_{j=0}^{\infty}d_{j}(t)(1-t)^{2j}\big((1-t)(q+t)\big)^{-\frac{n+2j-1}{2}}
	\I_{n+2j-1}\big(2\sqrt{(1-t)(q+t)}\tau \big)\\
& \quad +\sum_{n=0}^{\infty}C_{n}(t)
	\sum_{j=0}^{\infty}d_{j}(t)(1-t)^{2j}\big((1-t)(q+t)\big)^{-\frac{n+2j-1}{2}}
	\I_{n+2j+1}\big(2\sqrt{(1-t)(q+t)}\tau \big)\\
& \quad -\sum_{n=1}^{\infty}C_{1}(t)C_{n-1}(t)\\
	& \quad \quad \quad \times \sum_{j=0}^{\infty}d_{j}(t)
			 (1-t)^{2j}\big((1-t)(q+t)\big)^{-\frac{n-1+2j}{2}}
			\I_{n+2j-1}\big(2\sqrt{(1-t)(q+t)}\tau\big). 
\end{align*}
Here, we used (\ref{Bf}) and we remark that we are allowed to change the order of the differentiation and power series by (\ref{Bbound}). 
By Proposition \ref{Cm}, we have 
\begin{align*}
C_{n}(t)=\left \{ \begin{array}{ll}
			C_{1}(t)^{2}-(1-t)(q+1)I		& \text{if $n=2$}, \\
			C_{1}(t)C_{n-1}(t)-(1-t)(q+t)C_{n-2}(t)	& \text{if $n\geq 3$}.\\
			\end{array} \right. 
\end{align*}
Here, we remark that the operator $C_{n}(t)$ is a self-adjoint operator. 
By this relation, we have 
\begin{align*}
&\sum_{n=1}^{\infty}C_{1}(t)C_{n-1}(t)\\
	&\quad \times \sum_{j=0}^{\infty}d_{j}(t)
			 (1-t)^{2j}\big((1-t)(q+t)\big)^{-\frac{n-1+2j}{2}}
			\I_{n+2j-1}\big(2\sqrt{(1-t)(q+t)}\tau\big)\\
&=C_{1}(t)\sum_{j=0}^{\infty}d_{j}(t)
			 (1-t)^{2j}\big((1-t)(q+t)\big)^{-j}
			\I_{2j}\big(2\sqrt{(1-t)(q+t)}\tau\big)\\
	&\quad+C_{2}(t)\sum_{j=0}^{\infty}d_{j}(t)
			 (1-t)^{2j}\big((1-t)(q+t)\big)^{-\frac{1+2j}{2}}
			\I_{2j+1}\big(2\sqrt{(1-t)(q+t)}\tau\big)\\
	&\quad+(1-t)(q+1)\sum_{j=0}^{\infty}d_{j}(t)
			 (1-t)^{2j}\big((1-t)(q+t)\big)^{-\frac{1+2j}{2}}
			\I_{2j+1}\big(2\sqrt{(1-t)(q+t)}\tau\big)\\
	&\quad+\sum_{n=3}^{\infty}C_{n}(t)\sum_{j=0}^{\infty}d_{j}(t)
			 (1-t)^{2j}\big((1-t)(q+t)\big)^{-\frac{n-1+2j}{2}}
			\I_{n+2j-1}\big(2\sqrt{(1-t)(q+t)}\tau\big)\\
	&\quad+\sum_{n=1}^{\infty}C_{n}(t)\\
	&\quad \quad \quad \times \sum_{j=0}^{\infty}d_{j}(t)(1-t)^{2j}
			 \big((1-t)(q+t)\big)^{-\frac{n+2j-1}{2}}
			\I_{n+2j+1}\big(2\sqrt{(1-t)(q+t)}\tau\big). 
\end{align*}
To explain our calculation clearly, we put 
\begin{align*}
&(n=1)=C_{1}(t)\sum_{j=0}^{\infty}d_{j}(t)
			 (1-t)^{2j}\big((1-t)(q+t)\big)^{-j}
			\I_{2j}\big(2\sqrt{(1-t)(q+t)}\tau\big), \\
&(n=2-1)=C_{2}(t)\sum_{j=0}^{\infty}d_{j}(t)
			 (1-t)^{2j}\big((1-t)(q+t)\big)^{-\frac{1+2j}{2}}
			\I_{2j+1}\big(2\sqrt{(1-t)(q+t)}\tau\big), \\
&(n=2-2)=(1-t)(q+1)\sum_{j=0}^{\infty}d_{j}(t)
			 (1-t)^{2j}\big((1-t)(q+t)\big)^{-\frac{1+2j}{2}}
			\I_{2j+1}\big(2\sqrt{(1-t)(q+t)}\tau\big), \\
&(n=3-1)=\sum_{n=3}^{\infty}C_{n}(t)\sum_{j=0}^{\infty}d_{j}(t)
			 (1-t)^{2j}\big((1-t)(q+t)\big)^{-\frac{n-1+2j}{2}}
			\I_{n+2j-1}\big(2\sqrt{(1-t)(q+t)}\tau\big), \\
&(n=3-2)=\sum_{n=1}^{\infty}C_{n}(t)
	\sum_{j=0}^{\infty}d_{j}(t)(1-t)^{2j}
			 \big((1-t)(q+t)\big)^{-\frac{n+2j-1}{2}}
			\I_{n+2j+1}\big(2\sqrt{(1-t)(q+t)}\tau\big). 
\end{align*}
By calculating $(n=1)+(n=2-1)+(n=3-1)$, 
\begin{align*}
&\sum_{n=1}^{\infty}C_{1}(t)C_{n-1}(t)\\
	&\quad \times \sum_{j=0}^{\infty}d_{j}(t)
			 (1-t)^{2j}\big((1-t)(q+t)\big)^{-\frac{n-1+2j}{2}}
			\I_{n+2j-1}\big(2\sqrt{(1-t)(q+t)}\tau\big)\\
&=\sum_{n=1}^{\infty}C_{n}(t)\sum_{j=0}^{\infty}d_{j}(t)
			 (1-t)^{2j}\big((1-t)(q+t)\big)^{-\frac{n-1+2j}{2}}
			\I_{n+2j-1}\big(2\sqrt{(1-t)(q+t)}\tau\big)\\
	&\quad+(n=2-2)+(n=3-2). 
\end{align*}
Then, we have 
\begin{align*}
&\frac{\partial g}{\partial \tau}(\tau)-C_{1}(t)g(\tau)\\ \displaybreak[3]
&=\sum_{j=0}^{\infty}d_{j}(t)(1-t)^{2j}\big((1-t)(q+t)\big)^{-\frac{2j-1}{2}}
	\I_{2j-1}\big(2\sqrt{(1-t)(q+t)}\tau\big)\\
	&\quad+\sum_{n=0}^{\infty}C_{n}(t)
	\sum_{j=0}^{\infty}d_{j}(t)(1-t)^{2j}\big((1-t)(q+t)\big)^{-\frac{n+2j-1}{2}}
	\I_{n+2j+1}\big(2\sqrt{(1-t)(q+t)}\tau \big)\\
	&\quad -(n=2-2)-(n=3-2)\\
&=\sum_{j=0}^{\infty}d_{j}(t)(1-t)^{2j}\big((1-t)(q+t)\big)^{-\frac{2j-1}{2}}
	\I_{2j-1}\big(2\sqrt{(1-t)(q+t)}\tau\big)\\
	&\quad+\sum_{j=0}^{\infty}d_{j}(t)(1-t)^{2j}\big((1-t)(q+t)\big)^{-\frac{2j-1}{2}}
	\I_{2j+1}\big(2\sqrt{(1-t)(q+t)}\tau \big)-(n=2-2)\\
&=\sum_{j=0}^{\infty}d_{j}(t)(1-t)^{2j}\big((1-t)(q+t)\big)^{-\frac{2j-1}{2}}
	\I_{2j-1}\big(2\sqrt{(1-t)(q+t)}\tau\big)\\
	&\quad+\sum_{j=0}^{\infty}d_{j}(t)(1-t)^{2j}\bigg(1-\frac{q+1}{q+t}\bigg)\big((1-t)(q+t)\big)^{-\frac{2j-1}{2}}
	\I_{2j+1}\big(2\sqrt{(1-t)(q+t)}\tau \big)\\
&=\sum_{j=0}^{\infty}d_{j}(t)(1-t)^{2j}\big((1-t)(q+t)\big)^{-\frac{2j-1}{2}}
	\I_{2j-1}\big(2\sqrt{(1-t)(q+t)}\tau\big)\\
	&\quad-\sum_{j=0}^{\infty}d_{j}(t)(1-t)^{2j}\bigg(\frac{1-t}{q+t}\bigg)\big((1-t)(q+t)\big)^{-\frac{2j-1}{2}}
	\I_{2j+1}\big(2\sqrt{(1-t)(q+t)}\tau \big). 
\end{align*}
Here, we note that $C_{0}(t)=I$. 
Therefore, we have 
\begin{align*}
\frac{\partial g}{\partial \tau}(\tau)-C_{1}(t)g(\tau)=0. 
\end{align*}
Here, we used $\I_{-1}(\tau)=\I_{1}(\tau)$. 
This completes the proof. 
\end{proof}

\section{An alternative proof of the Bartholdi zeta function formula}
In this section, we give an alternative proof of the Bartholdi zeta function formula for a regular graph obtained in Section $4$.  

Since $\Delta_{X}$ is a self-adjoint bounded operator, there exists a unique spectral measure $E$ such that 
\begin{align*}
\Delta_{X}=\int_{\sigma(\Delta_{X})}\lambda dE(\lambda). 
\end{align*}
Here, $\sigma(\Delta_{X})$ stands for the spectrum of $\Delta_{X}$. 
Therefore, for $x_{0}, x \in V\!X$, we have 
\begin{align*}
K_{X}(\tau, x_{0}, x)=\int_{\sigma(\Delta)}\e^{-\tau \lambda} d\mu_{x_{0}, x}(\lambda). 
\end{align*}
Here, we denote $d\langle E(\lambda) \delta_{x_{0}}, \delta_{x} \rangle$ by $d\mu_{x_{0}, x}(\lambda)$. 
By applying the $G(t)$-transform and by easy culculation, for $0 < u < \frac{1}{\alpha(t)}$, we have 
\begin{align*}
G(t)\big(K_{X}(\tau, x_{0}, x)\big)(u)
&=G(t)\bigg( \int_{\sigma(\Delta)}\e^{-\tau \lambda} d\mu_{x_{0}, x}(\lambda) \bigg)(u) \\
&=\int_{\sigma(\Delta_{X})}\frac{u^{-2}-(q+t)(1-t)}{(q+t)(1-t)u+\frac{1}{u}-(q+1-\lambda)}d \mu_{x_{0}, x}(\lambda). 
\end{align*}
We remark that we are allowed to change the order of integrations in the above equation by Fubini's theorem. 
We also remark that if $0 < u < \frac{1}{\alpha(t)}$, then we have 
\begin{align*}
(q+t)(1-t)u+\frac{1}{u}-(q+1-\lambda)>0. 
\end{align*}
On the other hand, by Theorem \ref{heat kernel}, we have 
\begin{align*}
G(t)\big(K_{X}(\tau, x_{0}, x)\big)(u)
&=G(t)\bigg( \sum_{n=0}^{\infty}C_{n}(t)(x_{0}, x)
			\sum_{j=0}^{\infty}d_{j}(t) \e^{-(q+1)\tau} \\
			&\quad \times (1-t)^{2j}\big((1-t)(q+t)\big)^{-\frac{n+2j}{2}}
			\I_{n+2j}\big(2\sqrt{(1-t)(q+t)}\tau\big) \bigg)(u) \\
&=\sum_{n=0}^{\infty}C_{n}(t)(x_{0}, x)\sum_{j=0}^{\infty}d_{j}(t)(1-t)^{2j}u^{n+2j-1}. 
\end{align*}
Here, we used (\ref{G(t)f}) in the second equation. 
Therefore, we have 
\begin{align*}
G(t)(K_{X}(\tau, x_{0}, x))(u)
&=\sum_{n=0}^{\infty}C_{n}(x_{0}, x)u^{n-1}\\
&\quad-\frac{q-1+2t}{1-t}\sum_{n=2}^{\infty}
	\sum_{j=1}^{\lfloor \frac{n}{2} \rfloor}C_{n-2j}(t)(x_{0}, x)(1-t)^{2j}u^{n-1}. 
\end{align*}
We note that the following equality holds. 
\begin{align*}
\sum_{n=2}^{\infty}\sum_{j=1}^{\lfloor \frac{n}{2} \rfloor}C_{n-2j}(t)(x_{0}, x)(1-t)^{2j}u^{n}
&=\sum_{n=3}^{\infty}\sum_{j=1}^{\lceil \frac{n}{2} \rceil-1}(1-t)^{2j}C_{n-2j}(t)(x_{0}, x)u^{n}\\
	&\quad+C_{0}(t)(x_{0}, x)\frac{(1-t)^{4}u^{4}}{1-(1-t)^{2}u^{2}}+C_{0}(t)(x_{0}, x)(1-t)^{2}u^{2}. 
\end{align*}
Hence, we have 
\begin{align*}
&G(t)(K_{X}(\tau, x_{0}, x))(u)\\
&=\sum_{n=3}^{\infty}C_{n}(t)(x_{0}, x)u^{n-1}
	-\frac{q-1+2t}{1-t}\sum_{n=3}^{\infty}\sum_{j=1}^{\lceil \frac{n}{2} \rceil-1}
	(1-t)^{2j}C_{n-2j}(t)(x_{0}, x)u^{n-1}\\
	&\quad+\frac{1}{u}\big(C_{0}(t)(x_{0}, x)+C_{1}(t)(x_{0}, x)u+tC_{2}(t)(x_{0}, x)u^{2}\big)\\
	&\quad+u(1-t)C_{2}(t)(x_{0}, x)-(q-1+2t)\frac{(1-t)u}{1-(1-t)^{2}u^{2}}C_{0}(t)(x_{0}, x). 
\end{align*}
By the definition of $C_{m}^{\cbc}(t)$, we have 
\begin{align*}
&G(t)\big(K_{X}(\tau, x_{0}, x) \big)(u)\\
&=\frac{1}{u}C_{0}(t)(x_{0}, x)
+\frac{1}{u}\sum_{n=1}^{\infty}\big(C_{n}^{\cbc}(t)(x_{0}, x)-(1-t)R_{m}(t)(x_{0}, x)\big)u^{m}\\
&\quad+(1-t)u\big(C_{2}(t)-tD_{X}\big)(x_{0}, x)-(q-1)\frac{(1-t)^{2}u}{1-(1-t)^{2}u^{2}}C_{0}(t)(x_{0}, x). 
\end{align*}
Therefore, we have 
\begin{align*}
&\int_{\sigma(\Delta_{X})}\frac{u^{-2}-(q+t)(1-t)}{(q+t)(1-t)u+\frac{1}{u}-(q+1-\lambda)}d \mu_{x_{0}, x}(\lambda)\\
&=\frac{1}{u}C_{0}(t)(x_{0}, x)
+\frac{1}{u}\sum_{n=1}^{\infty}\big(C_{n}^{\cbc}(t)(x_{0}, x)-(1-t)R_{m}(t)(x_{0}, x)\big)u^{m}\\
&\quad+(1-t)u\big(C_{2}(t)-tD_{X}\big)(x_{0}, x)-(q-1)\frac{(1-t)^{2}u}{1-(1-t)^{2}u^{2}}C_{0}(t)(x_{0}, x). 
\end{align*}
This is equivalent to the following equation. 
\begin{align*}
&\frac{d}{du}\bigg\{\frac{q-1}{2}C_{0}(t)\log(1-(1-t)^{2}u^{2})
+\sum_{n=1}^{\infty}\frac{C_{n}^{\cbc}(t)-(1-t)R_{n}(t)}{n}u^{n}\\
&\quad+\frac{C_{2}(t)-tD_{X}}{2}(1-t)u^{2}\bigg\}(x_{0}, x)\\
&=\frac{d}{du} \int_{\sigma(\Delta_{X})}
	-\log\big(1-(q+1-\lambda)u+(1-t)(q+t)u^{2} \big) d\mu_{x_{0}, x}(\lambda)
\end{align*}
By integrating both sides from $0$ to $u$ and determining the integrating constant, we have 
\begin{align*}
&(1-(1-t)^{2})^{\frac{q-1}{2}\delta_{x_{0}}(x)}Z_{X}(u, t, x_{0}, x) 
\exp\bigg(-\sum_{n=3}^{\infty}\frac{(1-t)R_{m}(t)(x_{0}, x)}{m}\bigg)\\
&\quad \times \exp\bigg(\frac{\big[C_{2}(t)-tD_{X}\big](x_{0}, x)}{2}(1-t)u^{2}\bigg)\\
&=\exp\bigg(\int_{\sigma(\Delta_{X})}-\log\big(1-(q+1-\lambda)u+(1-t)(q+t)u^{2} \big) d\mu_{x_{0}, x}(\lambda)\bigg). 
\end{align*}
Therefore, we have
\begin{align*}
Z_{X}(u, t, x_{0}, x)
&=(1-(1-t)^{2}u^{2})^{-\frac{q-1}{2}\delta_{x_{0}}(x)}\\
&\quad\times \exp\bigg( \int_{\sigma(\Delta_{X})}-\log(1-(q+1-\lambda)u+(1-t)(q+t)u^{2})d\mu_{x_{0}, x}(\lambda)\bigg)\\
&\quad \times \exp\bigg(\frac{\big[tD_{X}-C_{2}(t)\big](x_{0}, x)}{2}(1-t)u^{2}\bigg)\\
&\quad \times \exp\bigg( \sum_{m=3}^{\infty}\frac{(1-t)R_{m}(t)(x_{0}, x)}{m}u^{m}\bigg). 		
\end{align*}
This gives an alternative proof of the Bartholdi zeta function formula obtained in Section $4$. 

\section*{Acknowledgment}
The author expresses gratitude to Professor Hiroyuki Ochiai for his many helpful comments.

\end{document}